\documentclass[12pt]{extarticle}
\usepackage[margin=0.3in, top=1in, bottom=1in]{geometry}
\usepackage[utf8]{inputenc}
\usepackage{amsmath,amssymb,amsthm,mathtools}
\usepackage{thmtools, thm-restate}
\usepackage{changepage}
\usepackage[T1]{fontenc}
\usepackage{eulervm,eucal,eufrak}
\usepackage{bbold}
\usepackage{concrete}
\usepackage{lipsum,hyperref,datetime2}
\usepackage[dvipsnames]{xcolor}
\hypersetup{
    colorlinks=true,
    filecolor=red,
    linkcolor=blue,
    citecolor=blue,
    urlcolor=blue,
}
\usepackage{graphicx}
\graphicspath{ {./graphics/} }
\usepackage[backend=biber,style=alphabetic]{biblatex} 
\bibliography{reference.bib} 
\newtheoremstyle{mystyle}
  {}
  {}
  {}
  {}
  {\bfseries}
  {}
  { }
  {\thmname{#1}\thmnumber{ #2}\thmnote{ (#3)}}
\numberwithin{equation}{section} 
\theoremstyle{mystyle}
\newtheorem{theorem}[equation]{Theorem}
\AtEndEnvironment{theorem}{\null\hfill$\diamond$}%

\AtEndEnvironment{convention}{\null\hfill$\diamond$}%

\AtEndEnvironment{proposition}{\null\hfill$\diamond$}%
\newtheorem{fact}[equation]{Fact}
\AtEndEnvironment{fact}{\null\hfill$\diamond$}%
\newtheorem{definition}[equation]{Definition}
\AtEndEnvironment{definition}{\null\hfill$\diamond$}%
\newtheorem{lemma}[equation]{Lemma}
\AtEndEnvironment{lemma}{\null\hfill$\diamond$}%

\AtEndEnvironment{corollary}{\null\hfill$\diamond$}%

\AtEndEnvironment{example}{\null\hfill$\diamond$}%

\AtEndEnvironment{remark}{\null\hfill$\diamond$}%
\newtheorem{notation}[equation]{Notation}
\AtEndEnvironment{notation}{\null\hfill$\diamond$}%
\let\oldproof\proof
\definecolor{my-dark-gray}{gray}{0.13}
\renewcommand{\proof}{\color{my-dark-gray}\oldproof}

\usepackage{tikz}
\usetikzlibrary{shapes.geometric, arrows}
\usetikzlibrary{shapes.multipart} 
\usetikzlibrary{scopes}
\usetikzlibrary{math}
\usetikzlibrary{decorations.markings,decorations.pathreplacing}
\usetikzlibrary{fadings}
\usepackage{tikz-cd}

\tikzset{
every picture/.style={line width=0.8pt, >=stealth,
                       baseline=-3pt,label distance=-3pt},
emptynode/.style={circle,minimum size=0pt, inner sep=0pt, outer
sep=0},
dotnode/.style={fill=black,circle,minimum size=2.5pt, inner sep=1pt, outer
sep=0},
small_dotnode/.style={fill=black,circle,minimum size=2pt, inner sep=0pt, outer
sep=0},
morphism/.style={fill=white,circle,draw,thin, inner sep=1pt, minimum size=15pt,
                 scale=0.8},
small_morphism/.style={fill=white,circle,draw,thin,inner sep=1pt,
                       minimum size=10pt, scale=0.8},
ellipse_morphism/.style args={#1}{fill=white,circle,draw,thin,inner sep=1pt,
                       minimum size=5pt, scale=0.8,
												ellipse, draw, rotate=#1},
coupon/.style={draw,thin, inner sep=1pt, minimum size=18pt,scale=0.8},
semi_morphism/.style args={#1,#2}{
                  fill=white,semicircle,draw,thin, inner sep=1pt, scale=0.8,
                  shape border rotate=#1,
                  label={#1-90:#2}},
regular/.style={densely dashed}, 
edge/.style={very thick, draw=green, text=black},
overline/.style={preaction={draw,line width=2mm,white,-}},
thin_overline/.style={preaction={draw,line width=#1 mm,white,-}},
thin_overline/.default=2,
thick_overline/.style={preaction={draw,line width=3mm,white,-}},
really_thick/.style={line width=3mm, gray},
boundary/.style={thick,  draw=blue, text=black},
ribbon/.style={line width=1.5mm, postaction={draw,line width=1mm,white}},
ribbon_u/.style args={#1,#2}{line width=#1mm, postaction={draw,line width=#2mm,white}},
cell/.style={fill=black!10},
subgraph/.style={fill=black!30},
midarrow/.style={postaction={decorate},
                 decoration={
                    markings,
                    mark=at position #1 with {\arrow{>}},
                 }},
midarrow/.default=0.5,
midarrow_rev/.style={postaction={decorate},
                 decoration={
                    markings,
                    mark=at position #1 with {\arrow{<}},
                 }},
midarrow_rev/.default=0.5,
block/.style={rectangle, rounded corners, text centered, draw=black, align=center}
}

\tikzstyle{block} = [rectangle, rounded corners, text centered, draw=black, align=center]

\tikzfading[name=fade inside, inner color=transparent!80, outer
color=transparent!10]

\title{Explicit Factorization of a Categorical Center}
\author{Jin-Cheng Guu and Ying Hong Tham}
\date{}

\begin{document}

\maketitle

\abstract{

  Given a braided fusion category $C$, it is well known that the
  natural map $C \boxtimes C^{bop} \to Z(C)$ from the square of
  $C$ to the (Drinfeld) categorical center $Z(C)$ is an
  equivalence if and only if $C$ is modular. However, it is not
  clear how to construct the inverse and the natural
  isomorphisms. In this work, we provide an explicit construction
  using insights from a specific quantum field theory, and
  explore how the equivalence fails for the degenerate cases.

}

\tableofcontents

\section{Introduction}

~\newline
\noindent \textbf{Ubiquity of Symmetries}

Symmetry has always been the central topic of pure mathematics.
The reason of its ubiquity is straightforward: Any symmetrical
object can be made much simpler by dividing out the extra
information, and any object with its original symmetry divided
out can be made more intuitive by recovering the extra
information. It is therefore useful, practically and
conceptually, to pass between both pictures.

~\newline
\noindent \textbf{Symmetries in Physics}

During the mid-$19$th century, the use of symmetries entered and
fundamentally changed theoretical physics. In particular, the
advent of Maxwell's equations and their further reduction by
$U(1)$-symmetry provided profound insights leading us to the
birth of Einstein's relativity and quantum mechanics. Later in
the $20$th century, the hidden symmetries of our universe were
further exploited by gauge theorists, who provided the celebrated
Standard Model.

~\newline
\noindent \textbf{Groups as Classical Symmetries}

Mathematically, the backbone of symmetries in gauge theory was
provided by \textit{groups} (collections of symmetries). By
understanding possible behaviors of groups, we can predict the
phenomena in an object (e.g. the space-time, a statistical model,
a molecule \cite{serre/finite-group-rep}.. etc) whose symmetries
are described by the group.

~\newline
\noindent \textbf{Higher Symmetries}

However, while successful, groups do not describe all kinds of
symmetries. The existence of the quantum groups is one of the
precursor of such insufficiency. In general, we need higher
symmetries, some of which are described by fusion categories,
which play an important role in modern mathematics, physics (both
theoretical and experimental) \cite{kong/anyon-condensation},
information theory and coding theory, quantum computing (both
theoretical and architectural).. etc.

~\newline
\noindent \textbf{Fusion Categories as Higher Symmetries}

In a nutshell, by viewing a finite group (and its group algebra)
as a $1$-categorical object, one can view a fusion category as an
analogous $2$-categorical object and a braided fusion category as
an analogous $3$-categorical object. As we climb the categorical
ladder, more information about symmetries is preserved.

~\newline
\noindent \textbf{Fusion Categories in Quantum Field Theory}

The current work is about the structure of a braided fusion
category. While we have mentioned its usefulness and its
abundance of applications, we wish to stress on a particular
application in quantum field theory and mathematics.

\noindent \textbf{}

The story started from knot theory. A knot is (an isotopy class
of) an embedding of a circle in an Euclidean $3$-space. While
both a circle and an Euclidean $3$-space are trivial, the
embeddings are notoriously hard to study. As an illustrative
example, try to tell if the following knots are isotopic to a
trivial embedding \cite{ochiai/hard-unknot}.

\noindent \textbf{}

\begin{center}
  \includegraphics[height=4cm]{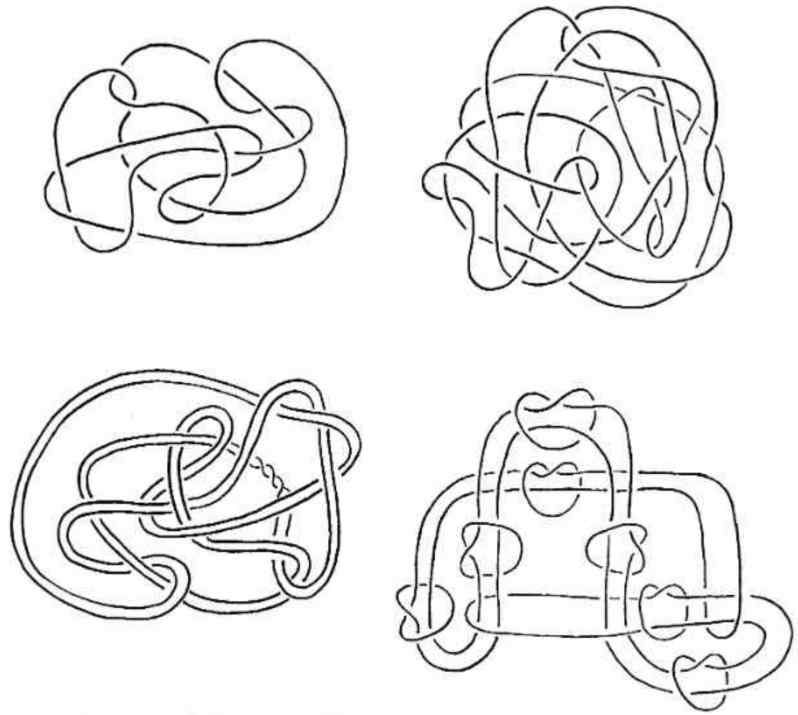}
\end{center}

\noindent \textbf{}

Despite its playfulness, it is beyond a fun brain-teaser. In
fact, knot theory relates deeply to theoretical physics
\cite{baez/knot-gravity} and modern number theory. It is
therefore of importance to develop a complete understanding
towards knot theory, which remains to be a wildly open subject
that currently keeps many mathematicians and physicists busy.

~\newline
\noindent \textbf{An Unexpected Invariant: Jones Polynomials}

One of the tremendous breakthrough was made in the mid $80$s by
the Fields medalist Vaughan Jones. His work on the Jones
polynomials unexpectedly \footnote{This amazing discovery won him
  a Fields medal in 1990.} connected knot theory and statistical
mechanics. The mysterious Jones polynomials were later explained
by Edward Witten in $1989$ as a special case of a larger
framework, the $3$-dimensional Witten-Reshetikhin-Turaev (WRT)
topological quantum field theory (TQFT), which is a quantized
Chern-Simons theory in dimension $3$.

~\newline
\noindent \textbf{Jones Polynomials as a Special Case of a $3$D TQFT}

Enter the braided fusion categories. They are the algebraic input
of the WRT TQFTs. However, we need them to be \textit{modular}
(or called \textit{non-degenerate}) in order to define the
theory. The WRT theory is itself a vast framework that still
keeps mathematicians busy, for its deep relationship with the
theory of integer primes \cite{morishita/knot-and-prime}, (mock)
modular forms \cite{cheng/q-inv-and-mock-theta}, topological
quantum computations \cite{kitaev/qc-by-anyons}..

~\newline
\noindent \textbf{A Larger Framework: Crane-Yetter $4$D TQFT}

Despite its vastness, there is a larger framework behind it, the
$4$-dimensional Crane-Yetter topological quantum field theory
(CY). Several evidences \footnote{The full result has been long
  expected but left unwritten partly because a rigorous proof
  needs to establish an equivalence between $\partial CY$ and
  $WRT$ as $2$-functors, which is quite technical.} have been
found that WRT is the boundary theory of CY whenever the
algebraic input is non-degenerate
\cite{barrett/observables-in-tv-and-cy}
\cite{fac-homo--kirillov-tham} \cite{tham/phd-thesis}, making CY
more flexible and potentially more powerful. Therefore, it is
natural to understand how much more power the degeneracy gives to
the CY model.

~\newline
\noindent \textbf{CY - WRT = ?}

With the excision formula of CY, the degeneracy leads to the
non-triviality of CY (in dimension $2$) precisely because of the
following statement.

\begin{fact} 
  Let $C$ be a braided fusion category, and let $F$ be the
  natural functor from the square of $C$ to the (Drinfeld)
  categorical center
  $$C \boxtimes C^{bop} \xrightarrow{F} Z(C)$$
  $$X \boxtimes Y \mapsto (X \otimes Y, c_{-,X} \otimes (c^{\text{-}1})_{Y,-})$$
  where $c_{-,\star}$ denotes the braiding of $C$. Then $F$ is an
  equivalence of categories if and only if $C$ is non-degenerate.
\end{fact}

\noindent The degeneracy fails the equivalence and makes the
algebra richer. But this also makes the TQFT stronger as it
preserves more topological information.

\noindent \textbf{}

It is then urgent to analyze how the degeneracy fails the
equivalence. The main crux is to look at how $F$ and its inverse
fails to compose to functors that are equivalent to the identity
functors. However, to the best of our knowledge, the explicit
construction of the inverse map and the natural equivalences are
not known. The main result provided in the current paper is an
explicit construction of the inverse and the witnessing natural
isomorphisms. Applications and explicit calculations will come in
future work.

~\newline
\noindent \textbf{Inspiration from Topology}

\noindent The main observation of the current work is based on
the following (surprising) fact.

\begin{fact} \cite{fac-homo--kirillov-tham} Let $C$ be a
  premodular category and $\Sigma$ be the cylinder
  $S^{1} \times I$. Then there is an equivalence of categories
  $CY_{C}(\Sigma) \simeq Z(C)$ where $Z(C)$ denotes the Drinfeld
  center.
\end{fact}

This fact is interesting because it gives a simple topological
interpretation of the grand idea - the Drinfeld center
construction \footnote{Such construction led to the construction
  of quantum groups, winning him a Fields medal in 1990.}. Using
this fact, we gained several insights about the Drinfeld center:

\begin{enumerate}
  \item Besides the natural monoidal structure, there is another
        hidden tensor product for $Z(C)$, namely the reduced
        tensor product $\overline{\otimes}$ \cite{reduced--tham}.
  \item It admits natural generalizations to all (oriented
        punctured) surfaces, namely the categorical center of
        higher genera \cite{guu/higher-genera-center}.
  \item It hints what an explicit inverse map of $F$ and the
        witnessing natural transformations would be. In
        hindsight, the explicit construction seems to be
        out-of-reach if one does not consider its topological
        picture provided by the Crane-Yetter model. This serves
        as the key idea of the paper.
\end{enumerate}

\section{Prerequisites}

We fix an algebraically closed field $\mathbb{k}$ with
characteristic $0$.

\subsection{Premodular Categories}

We will define (pre)modular categories assuming familiarity with
a fusion category, a braided category, ribbon structure, and the
(Drinfeld) categorical center. A complete and recommended source
is \cite{egno/tensor-cats}. For definitions written in a
dictionary-style starting from ``scratch'' (additive categories
and abelian categories), please refer to
\cite[appendix]{guu/higher-genera-center}. Other useful sources
are \cite{kirillov/mtc}, \cite{quantum-group--kassel},
\cite{turaev-qiok-3-manifolds}.

\begin{definition}[Braided Fusion Category]
  A \textit{braided fusion category} is a braided category whose
  underlying monoidal category is a fusion category.
\end{definition}

\begin{definition}[Muger center]
  Given a braided fusion category $C$ with braided structure
  $c_{-,\star}$, we say an object $X$ in $C$ is
  \textit{transparent} (and otherwise \textit{opaque}) if
  $$c_{-,X} \circ c_{X,-} = id_{X,-}.$$

  We define the Muger center $Mu(C)$ of $C$ to be the full tensor
  subcategory of $C$ consisting of transparent objects. Note that
  in some other literature, the Muger center is also called a
  Muger centralizer or an $E_{2}$-center.
\end{definition}

Recall that if $c_{-,\star}$ is a braided structure of a braided
fusion category $C$, then $c_{\star,-}^{\text{-}1}$ is also a
braided structure for the underlying fusion category. This
produces an opposite braided fusion category, which we denote by
$C^{bop}$. Directly by the definition of a (Drinfeld) categorical
center, there is a tautological functor from
$ C \boxtimes C^{bop}$ to $Z(C)$.

\begin{definition}[Tautological functor $F$]
  Given a braided fusion category $C$, there is a natural functor
  $C \boxtimes C^{bop} \xrightarrow{F} Z(C)$ that maps each object
  $X \boxtimes Y$ to
  $(X \otimes Y, c_{-,X} \otimes c_{Y,-}^{-1})$ and each morphism
  $(f \boxtimes g)$ to $(f \otimes g)$.
\end{definition}

\begin{definition}[Factorizable category]
  Given a braided fusion category $C$, if its tautological
  functor $F$ is an equivalence of categories, we say that $C$ is
  \textit{factorizable}, and call any of its inverse functor
  $F^{\text{-}1}$ a factorization of the Drinfeld center $Z(C)$.
\end{definition}

Notice that the structure of $Z(C)$ is in general pretty opaque.
For example, even the fusion ring of $Z(C)$ is hard to identify.
Factorizability reduces the complexity of $Z(C)$ to that of $C$.

\begin{definition}[Premodular Category]
  A premodular category is a ribbon fusion category
  (equivalently, a braided fusion category equipped with a
  spherical structure).
\end{definition}

\begin{definition}[Complete set of simple
  objects]\label{def/complete-set-of-simples}
  Let $C$ be a premodular category. By a complete set of simple
  objects $O(C)$ we mean a set $O(C) = \{i, j, \ldots\}$ of
  simple objects in $C$ that exhausts all simple types and that
  satisfies $(i \neq j) \Rightarrow (i \not\simeq j)$. Define its
  dual set to be
  $$O(C)^{\star} = \{ i^{\star} \,|\, i \in O(C)\},$$
  where $i^{\star}$ denotes the (left) dual object of $i$.
\end{definition}

\noindent Notice that by the axiom of premodular category, any
$O(C)$ is a finite set. From now on, we assume that any
premodular category $C$ comes with a fixed complete set of simple
objects $O(C)$.

\begin{definition}[$S$-matrix]
  Let $C$ be a premodular category with the braided structure
  $c$. The $S$-matrix of $C$ is defined by
  $$
  S := (s_{XY})_{X,Y \in \mathcal{O}(C)}
  $$
  where $s_{XY} = Tr(c_{Y,X}c_{X,Y}) \in \mathbb{k}$, where $Tr$
  denotes the (left) quantum trace that depends on the spherical
  structure of $C$.
\end{definition}

\begin{definition}[Modular Category]\cite[8.13.14]{egno/tensor-cats}
  A modular category is a premodular category $C$ whose
  $S$-matrix is non-degenerate.
\end{definition}

\begin{fact}[Characterization of Modularity]\cite[8.20.12 and 8.19.3]{egno/tensor-cats}
  The following conditions are equivalent for a premodular
  category $C$:
  \begin{enumerate}
    \item $C$ is modular.
    \item $Mu(C) \simeq (Vect.)$
    \item $C$ is factorizable.
  \end{enumerate}
\end{fact}

Surprisingly, the fact indicates that modularity and
factorizability are equivalent for premodular categories, so as a
consequence modularity reduces the complexity of $Z(C)$. While
this is desirable from the algebraic point of view, it is not the
case from the topological point of view: The power of the
topological quantum field theory is largely reduced by modularity
exactly due to this fact.

\subsection{Graphical Calculus}

We will use the technique of graphical calculus
(\cite{kirillov/mtc} and \cite{quantum-group--kassel}) while
dealing with premodular categories.

\begin{center}
  \includegraphics[height=6cm]{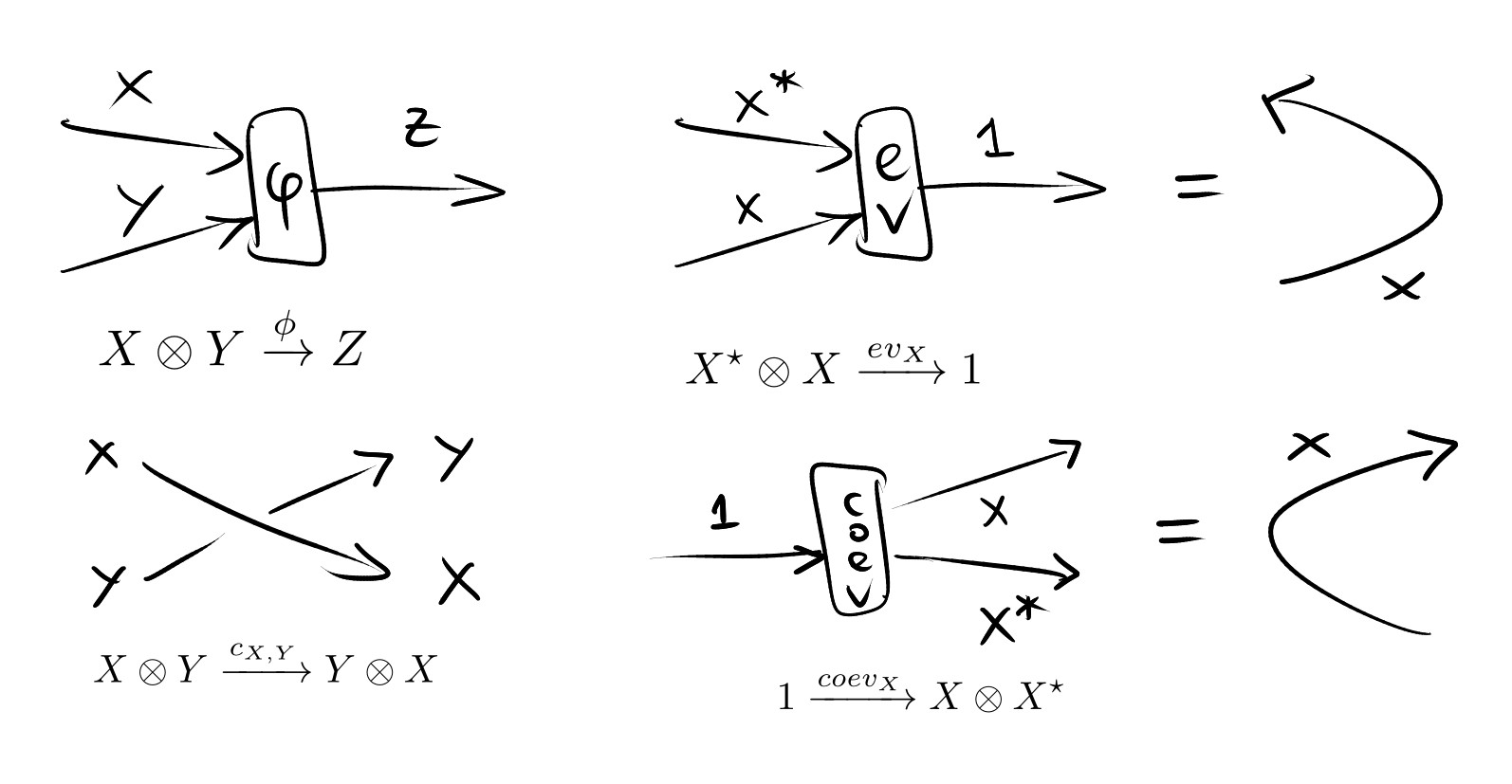}
\end{center}

\noindent An advantage of this is that many equalities among
morphisms can be proved graphically, thanks to the work of
Reshetikhin and Turaev \cite{RT-ribbon-graphs} \cite[Theorem
2.3.10]{kirillov/mtc}. For example, to prove
$$eval_{Y} \circ c_{X,Y} \circ c_{X,Y^{\star}} \circ c_{X,Y} \circ coev_{Y} = c_{X,Y},$$
\begin{center}
  \includegraphics[height=2cm]{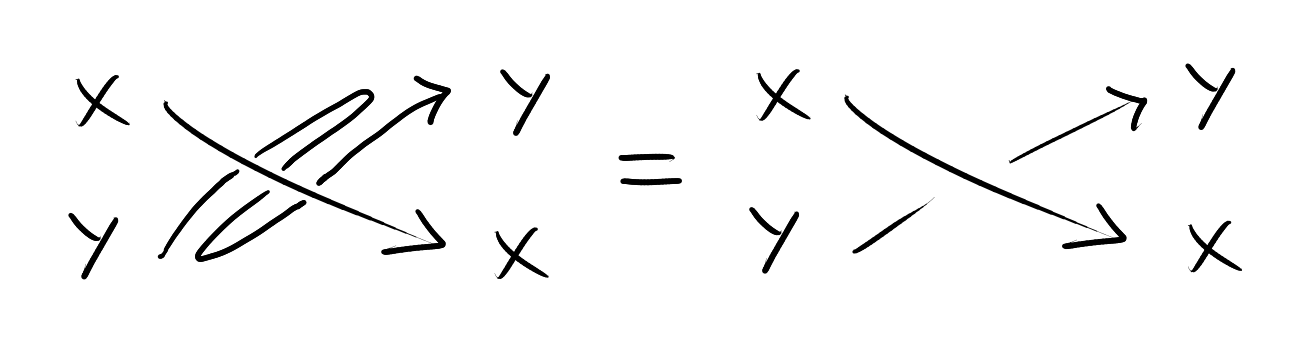}
\end{center}
it suffices to establish an isotopy of ribbon tangles, which is a
trivial task, and translate the procedure back into the equations
in the syntactic equations. Such feature of graphical calculus
provides sophisticated quantum link invariants (e.g. Jones
polynomials). An interesting exercise left for the unconvinced
reader is to turn all graphical equations in this paper into
syntactic equations.

In the rest of the section, we provide some useful
lemmas and notations for graphical calculus.

\begin{lemma}\label{lemma/dual-of-hom-space}
  Let $C$ be a premodular category with spherical structure $a$.
  Let $X, Y$ be $C$-objects. Define a pairing of
  $\mathbb{k}$-linear spaces
  $$Hom_{C}(X,Y) \otimes Hom_{C}(Y,X) \xrightarrow{(,)} \mathbb{k}$$
  that sends $\phi \otimes \psi$
  to $$Tr(\psi \circ \phi) = eval_{X} \circ ((a_{X} \circ \psi \circ \phi) \otimes 1_{X^{\star}}) \circ coev_{X} \in End_{C}(\mathbb{1}) \simeq \mathbb{k}.$$
  Then the pairing is nondegenerate by the semisimplicity of $C$,
  identifying the linear space with its linear dual
  $Hom_{C}(Y,X) \simeq Hom_{C}(X,Y)^{\star}$
\end{lemma}

\noindent Define the Casimir element
$$\omega_{X,Y} := \Sigma_{i} \phi_{i} \otimes \phi^{i} \in Hom_{C}(X,Y) \otimes Hom_{C}(Y,X)$$
where the $\phi_{i}$'s is any basis of the former multiplicand
and the $\phi^{i}$'s is its dual basis under the identification
given in \ref{lemma/dual-of-hom-space}. Graphically, we use dummy
variables $\phi$ and $\phi^{\star}$ as a short-hand notation:

\begin{center}
  \includegraphics[height=1.5cm]{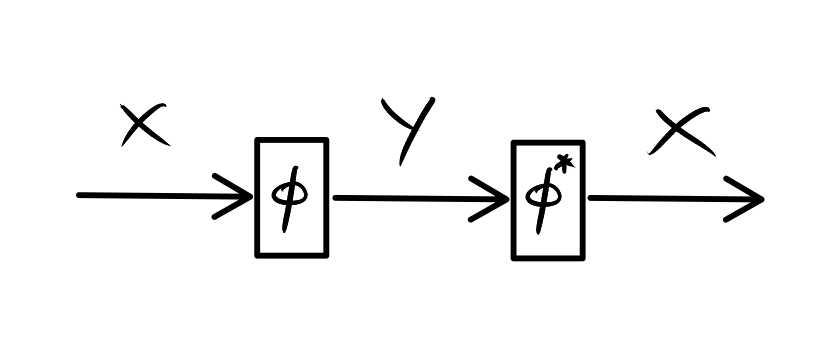}
\end{center}

\begin{lemma} \label{lemma/fundamental-lemma-of-Omega} Let $C$ be
  a premodular category and $W$ be a $C$-object. Then
  $$1_{W} = \Sigma_{i \in \mathcal{O}(C)} \Sigma_{l} dim(i) \phi^{l} \circ \phi_{l},$$
  where the $\phi_{l}$'s and the $\phi^{l}$'s form a pair of dual
  bases for the vector spaces $Hom_{C}(X,Y)$ and $Hom_{C}(Y,X)$
  respectively, and $dim(i)$ denotes the (left) quantum trace of
  $id_{i}$.
\end{lemma}

\begin{notation}[regular color]\label{def/Omega} Let $C$ be a
  premodular category and $O(C)$ a complete (up to isomorphism)
  set of simple objects of $C$. We use $\Omega$ in the graphics
  to represent the regular color
  $\oplus_{i \in \mathcal{O}(C)} dim(i) id_{i}: i \to i$. We also
  denote $dim(\Omega)$ by
  $\Sigma_{i \in \mathcal{O}(C)} dim(i)^{2}$, which is nonzero
  \cite{eno/fusion-cats}.
\end{notation}

\noindent With this shorthand notation $\Omega$, we can present
the lemma graphically by

\begin{center}
  \includegraphics[height=2cm]{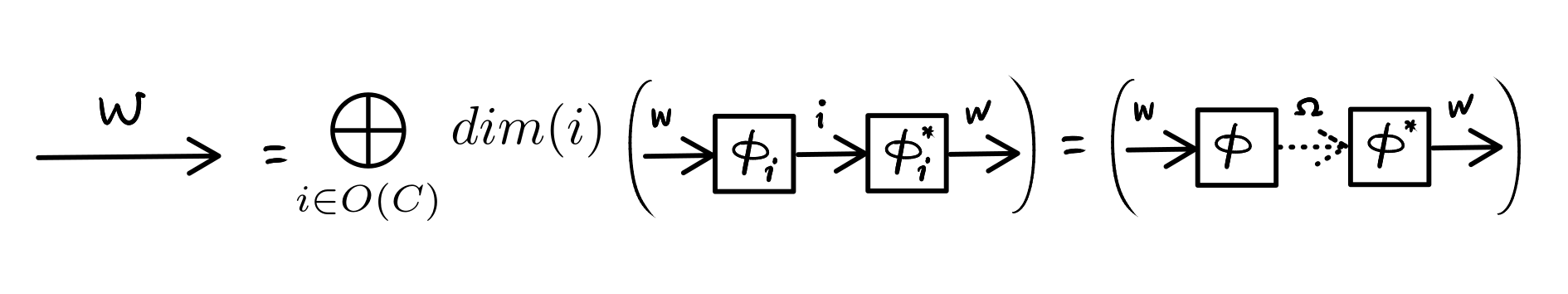}
\end{center}

\begin{lemma}[Sliding lemma]\label{lemma/sliding-lemma}
  Let $C$ be a premodular category. Then the following morphisms
  are all equal, where $\Omega$ is the shorthand notation given
  in \ref{def/Omega}.
  \begin{center}
    \includegraphics[height=2.5cm]{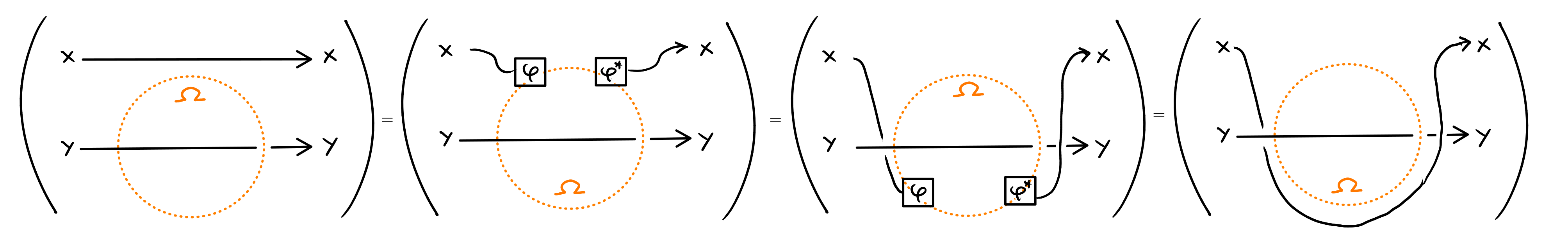}
  \end{center}
\end{lemma}

Heuristically, the moral of this lemma is that $\Omega$ protects
everything ``inside'' it by making it transparent.

\begin{lemma}[Censorship of Opacity]\label{lemma/censorship-of-opacity}\cite[(3.1.19)]{kirillov/mtc}
  Let $C$ be a modular category, $i$ a simple $C$-object, and
  $\lambda = dim(\Omega)\delta_{i,1}$, we have the following
  equality.
  \begin{center}
    \includegraphics[height=2cm]{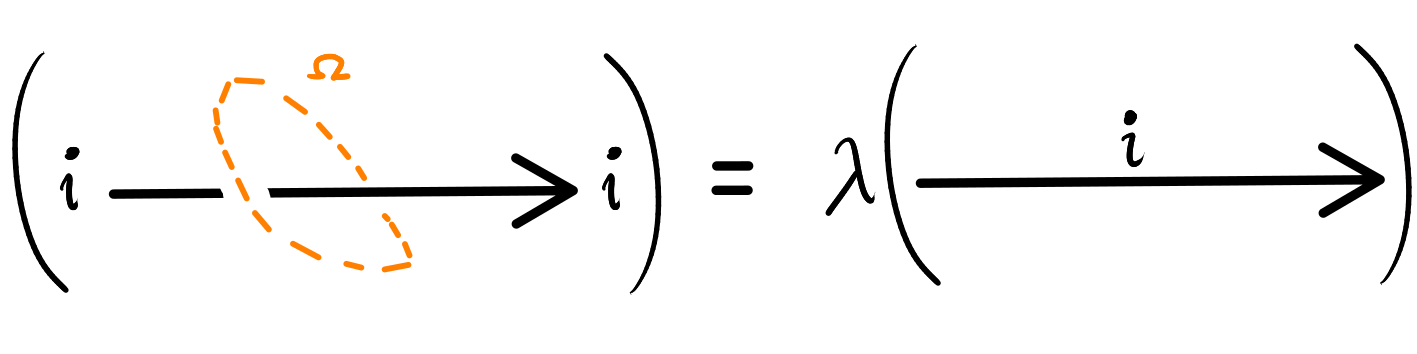}
  \end{center}
\end{lemma}

\section{Main Result}

For any premodular category $C$, we aim to construct a functor
$$Z(C) \xrightarrow{G} C \boxtimes C^{bop},$$ and prove it
an inverse functor for $C \boxtimes C^{bop} \xrightarrow{F} Z(C)$
in the case that $C$ is modular by constructing explicit natural
isomorphisms.

Throughout this section, we fix a premodular category $C$, fix a
complete set of simple objects $O(C)$ and its dual $O(C)^{\star}$
(\ref{def/complete-set-of-simples}). With a $C$-object $X$ fixed,
$Hom_{C}(X,i^{\star})$ is a finite dimensional vector space over
$\mathbb{C}$ with a natural nondegenerate pairing
\ref{lemma/dual-of-hom-space} with $Hom_{C}(i^{\star},X)$. Pick
and fix an arbitrary basis
$X[i] = \{\alpha_{i,1}, \ldots, \alpha_{i,l_{i}}\}$ for the
former space, and form its dual basis
$X[i]^{\star} = \{\alpha_{i}^{1}, \ldots \alpha_{i}^{l_{i}}\}$ in
the latter. We will drop the super/subfix when there is little
danger of confusion. We also identify $x$ and its (left) double
dual $x^{\star \star}$ by the spherical structure of $C$.

\subsection{Construction}

\begin{definition}[The coupling morphism $\Gamma_{i,(X,\gamma)}$]
  Let $(X,\gamma)$ be an object of $Z(C)$. For each $i \in O(C)$,
  define the $C$-morphism $\Gamma_{i,(X,\gamma)}$ to be the
  product of $\frac{1}{dim(\Omega)}$ and the following morphism

  \begin{center}
    \begin{tikzpicture}
      \draw[->] (0,0.7) -- (2.2,0.7)
      node[pos=0,left] {$i$} node[pos=1,right] {$i$};
      \draw[->] (0,0) -- (2.2,0)
      node[pos=0,left] {$X$} node[pos=1,right] {$X$};
      \node[dotnode] (a) at (0.8,0) {};
      \node[emptynode] (b) at (0.6,1) {};
      \node[emptynode] (c) at (1.3,0.7) {};
      \node[emptynode] (d) at (1.4,-0.3) {};
      \draw[overline,regular] (a) .. controls +(120:1cm) and +(180:0.2cm) ..
      (b) .. controls +(0:0.2cm) and +(120:0.4cm) ..
      (c) .. controls +(-60:1cm) and +(0:0.2cm) ..
      (d) .. controls +(180:0.2cm) and +(-60:0.4cm) .. (a);
      \draw[overline] (1,0.7) -- (2,0.7);
      \draw[overline] (1,0) -- (2,0);
      \node at (0.6,-0.2) {\small $\gamma$};
      \node at (0.9,1.2) {\small $\Omega$};
    \end{tikzpicture}
  \end{center}
\end{definition}

\noindent By axiom, $C$ is an abelian category, so there is a
canonical object $I_{i,(X,\gamma)}$ and two canonical maps
$(i \otimes X) \xrightarrow{\twoheadrightarrow} I_{i,(X,\gamma)}$
and $I_{i,(X,\gamma)} \xrightarrow{\subseteq} (i \otimes X)$ such
that
$\Gamma_{i,(X,\gamma)} = (\subseteq) \circ (\twoheadrightarrow)$.
Both canonical maps depend on $i$ and $(X,\gamma)$. However for
simplicity we often omit mentioning the dependence. Notice also
that $\Gamma_{i,(X,\gamma)}^{2} = \Gamma_{i,(X,\gamma)}$ by the
tensoriality of $\gamma$ and the sliding lemma.

\begin{definition}[The functor $G$]
  Define $G$ to be the functor that sends any object $(X,\gamma)$
  in $Z(C)$ to
  $$\bigoplus_{i \in O(C)} i^{\star} \boxtimes I_{i,(X,\gamma)},$$
  \noindent and any morphism
  $(X,\gamma) \xrightarrow{\phi} (Y,\beta)$ to
  $$\bigoplus_{i} 1_{i^{\star}} \boxtimes (\twoheadrightarrow \circ \Gamma_{i,(Y,\beta)} \circ (1_{i} \otimes \phi) \circ \Gamma_{i,(X,\gamma)} \circ \subseteq).$$
\end{definition}

\noindent
The construction was motivated by the Crane-Yetter theory and the
construction of the categorical center of higher genera
\cite{guu/higher-genera-center}. Notice what while its existence
was known by (Frobenius-Perron) dimension argument
\cite{egno/tensor-cats}, this construction is new and is expected
to provide insight in the difference between two topological
quantum field theories, the Witten-Reshetikhin-Turaev theory and
the Crane-Yetter theory.

We will further construct four natural transformations
$$1 \xrightarrow{d} GF,\quad GF \xrightarrow{q} 1,\quad 1 \xrightarrow{b} FG,\quad FG \xrightarrow{p} 1.$$
and argue that they witness $FG \simeq 1$ and $GF \simeq 1$ where
$C$ is modular.

\begin{definition}[The transformations $d$ and $q$]
  To construct the natural transformation $1 \xrightarrow{d} GF$,
  it suffices to construct a morphism in $C \boxtimes C^{bop}$
  for each object $X \boxtimes Y$. We thus define
  $$d = d_{X \boxtimes Y} := \sum_{i \in O(C)}d_{i} := \sum_{i \in O(C)}\frac{1}{\sqrt{dim(i)}}\sum_{k=1}^{|X[i]|}d_{i}(k),$$
  where $d_{i}(k)$ denotes the following morphism:

  \begin{center}
    \begin{tikzpicture}
      \node (y) at (0,0) {$Y$};
      \node at (0,0.8) {$\boxtimes$};
      \node (x) at (0,1.6) {$X$};
      \node (I) at (4.5,0) {$I_{i, F(X \boxtimes Y)}$};
      \node at (4,0.8) {$\boxtimes$};
      \node (i8) at (4,1.6) {$i^*$};
      \node (y2) at (2,-0.8) {\small $Y$};
      \node at (2,-0.4) {\small $\otimes$};
      \node (x2) at (2,0) {\small $X$};
      \node at (2,0.4) {\small $\otimes$};
      \node (i) at (2,0.8) {\small $i$};
      \draw[->] (x) -- (i8)
        node[pos=0.5,above] {\small $\alpha_{i,k}$};
      \draw[->] (y) to[out=-20,in=180] (y2);
      \draw[->>] (2.8,0) -- (I);
      \draw[decorate,decoration={brace,amplitude=6pt,mirror,raise=2pt}]
        (2.3,-0.8) -- (2.3,0.8);
      \draw 
        (i) .. controls +(-170:1.2cm) and +(170:1.2cm) .. (x2);
      \node[dotnode] at (1.3,0.15) {};
      \node at (1.3,-0.1) {\tiny $\alpha_i^k$};
      \node[dotnode] at (1.05,0.4) {};
      \node at (0.7,0.4) {\tiny coev};
    \end{tikzpicture}
  \end{center}

  Similarly, define the natural transformation
  $GF \xrightarrow{q} 1$ to be the sum
  $$q = q_{X \boxtimes Y} := \sum_{i \in O(C)}q_{i} := \sum_{i \in O(C)}\frac{1}{\sqrt{dim(i)}}\sum_{k=1}^{|X[i]|}q_{i}(k),$$
  where $q_{i}(k)$ denotes the following morphism:

  \begin{center}
    \begin{tikzpicture}
      \node (y) at (4,0) {$Y$};
      \node at (4,0.8) {$\boxtimes$};
      \node (x) at (4,1.6) {$X$};
      \node (I) at (0,0) {$I_{i_* F(X \boxtimes Y)}$};
      \node at (0,0.8) {$\boxtimes$};
      \node (i8) at (0,1.6) {$i^*$};
      \node (y2) at (2,-0.8) {\small $Y$};
      \node at (2,-0.4) {\small $\otimes$};
      \node (x2) at (2,0) {\small $X$};
      \node at (2,0.4) {\small $\otimes$};
      \node (i) at (2,0.8) {\small $i$};
      \node at (1.1,0) {$\subseteq$};
      \draw[->] (y2) to[out=0,in=-160] (y);
      \draw[->] (i8) -- (x)
        node[pos=0.5,above] {\small $\alpha_i^k$};
      \draw[decorate,decoration={brace,amplitude=6pt,raise=2pt}]
      (1.7,-0.8) -- (1.7,0.8);
      \draw (x2) .. controls +(10:1.2cm) and +(-10:1.2cm) .. (i);
      \node[dotnode] at (2.75,0.18) {};
      \node at (2.8,-0.05) {\tiny $\alpha_{i,k}$};
      \node[dotnode] at (2.95,0.4) {};
      \node at (3.2,0.4) {\tiny ev};
    \end{tikzpicture}
  \end{center}
\end{definition}

\noindent Notice that while $d_{i}(k)$ and $q_{i}(k)$ depend on
the choice $X[i]$, the morphisms $d_{i}$ and $q_{i}$ do not.

\begin{definition}[The transformations $b$ and $p$]
  To construct the natural transformation $1 \xrightarrow{b} FG$,
  it suffices to construct a morphism in $Z(C)$ from each object
  $(X,\gamma)$ to
  $FG(X,\gamma) = \bigoplus_{i\in O(C)} i^{\star} \otimes i \otimes X$.
  We thus define
  $$b = b_{(X,\gamma)} := \sum_{i \in O(C)}b_{i}$$
  where $b_{i}$ denotes the product of $\sqrt{dim(i)}$ and the
  following morphism:

  \begin{center}
    \begin{tikzpicture}
      \node (x) at (0,-0.8) {$X$};
      \node (I) at (5.3,-0.4) {$I_{i,(X,\gamma)}$};
      \node (i82) at (5,0.8) {$i^*$};
      \node at (5,0.2) {$\otimes$};
      \node (x2) at (2,-0.8) {$X$};
      \node at (2,-0.4) {$\otimes$};
      \node (i) at (2,0) {$i$};
      \node at (2,0.4) {$\otimes$};
      \node (i8) at (2,0.8) {$i^*$};
      \draw[->] (x) -- (x2);
      \draw[->] (i8) -- (i82);
      \draw[->] (2.8,-0.4) -- (I)
        node[pos=0.5,below] {\tiny $(\twoheadrightarrow) \circ \Gamma_{i,(X,\gamma)}$};
      \draw[decorate,decoration={brace,amplitude=6pt,mirror,raise=2pt}]
        (2.3,-0.8) -- (2.3,0);
      \node[dotnode] (coev) at (1,0.4) {};
      \node at (0.6,0.4) {\tiny coev};
      \draw (coev) to[out=70,in=180] (i8);
      \draw (coev) to[out=-70,in=180] (i);
    \end{tikzpicture}
  \end{center}

  Similarly, define the natural transformation
  $FG \xrightarrow{p} 1$ to be the sum
  $$p = p_{(X,\gamma)} := \sum_{i \in O(C)}p_{i}$$
  where $p_{i}$ denotes the product of $\sqrt{dim(i)}$ and the
  following morphism:

  \begin{center}
    \begin{tikzpicture}
      \node (x) at (5,-0.8) {$X$};
      \node (I) at (0,-0.4) {$I_{i,(X,\gamma)}$};
      \node (i82) at (0,0.8) {$i^*$};
      \node at (0,0.2) {$\otimes$};
      \node (x2) at (3,-0.8) {$X$};
      \node at (3,-0.4) {$\otimes$};
      \node (i) at (3,0) {$i$};
      \node at (3,0.4) {$\otimes$};
      \node (i8) at (3,0.8) {$i^*$};
      \draw[->] (x2) -- (x);
      \draw[->] (i82) -- (i8);
      \draw[->] (I) -- (2.2,-0.4)
        node[pos=0.5,below] {\tiny $\Gamma_{i,(X,\gamma)} \circ \subseteq$};
      \draw[decorate,decoration={brace,amplitude=6pt,raise=2pt}]
        (2.7,-0.8) -- (2.7,0);
      \node[dotnode] (ev) at (4,0.4) {};
      \node at (4.3,0.4) {\tiny ev};
      \draw (ev) to[out=110,in=0] (i8);
      \draw (ev) to[out=-110,in=0] (i);
    \end{tikzpicture}
  \end{center}
\end{definition}

It requires some effort to check that so defined transformations
$b$ and $p$ are indeed morphisms in $Z(C)$. We prove that in the
following lemma.

\begin{lemma}
  Given an $Z(C)$-object $(X,\gamma)$, the definition of the
  natural transformations $b = b_{(X,\gamma)}$ and
  $p = p_{(X,\gamma)}$ are indeed morphisms in $Z(C)$.
\end{lemma}
\begin{proof}
  By the definition of $Z(C)$, it suffices to show that $b$ and
  $d$ respect the half-braidings $\gamma$ and
  $c \otimes c^{\text{-}1}$. We provide a graphical proof for
  this fact.

  \noindent Since any proof for $b$ also works similarly for $p$,
  so we shall only prove for $b$. By definition of $Z(C)$, it
  suffices to prove the following equality (functorial in
  $Z \in Obj(C)$)

  \begin{center}
    \begin{tikzpicture}
      \begin{scope}[shift={(0,0)}]
      \node[emptynode] (e) at (0.6,0.7) {};
      \draw[->] (-2,1.6)
        .. controls +(0:2cm) and +(120:0.5cm) .. (e)
        .. controls +(-60:0.5cm) and +(180:0.5cm) .. (2,-1)
        node[pos=1,right] {$Z$};
      \draw[->,overline] (-2,0) -- (2,0)
        node[pos=1,right] {$X$};
      \draw[->,overline] (0,0.5) -- (2,0.5)
        node[pos=1,right] {$i$};
      \draw[overline] (0,0.5)
        .. controls +(180:0.8cm) and +(180:0.8cm) .. (0,1);
      \draw[overline] (0,1) -- (2,1);
      \node[dotnode] (ga) at (0,0) {};
      \node at (-0.15,-0.15) {\tiny $\gamma$};
      \node[emptynode] (a) at (0.2,0.5) {};
      \draw[overline,regular] (ga)
        .. controls +(120:0.8cm) and +(120:0.5cm) .. (a)
        .. controls +(-60:0.8cm) and +(-60:0.5cm) .. (ga);
      \draw[overline] (0.2,0) -- (1,0);
      \draw[overline] (0,0.5) -- (1,0.5);
      \draw[overline] (-2,1.6)
        .. controls +(0:2cm) and +(120:0.5cm) .. (e);
      \end{scope}
      \node at (3.7,0.2) {$=$};
      \begin{scope}[shift={(7,0)}]
      \draw[->,overline] (-2,0) -- (2,0)
        node[pos=1,right] {$X$};
      \draw[->,overline] (0,0.5) -- (2,0.5)
        node[pos=1,right] {$i$};
      \draw[overline] (0,0.5)
        .. controls +(180:0.8cm) and +(180:0.8cm) .. (0,1);
      \draw[overline] (0,1) -- (2,1);
      \node[dotnode] (ga) at (0,0) {};
      \node at (-0.15,-0.15) {\tiny $\gamma$};
      \node[emptynode] (a) at (0.2,0.5) {};
      \draw[overline,regular] (ga)
        .. controls +(120:0.8cm) and +(120:0.5cm) .. (a)
        .. controls +(-60:0.8cm) and +(-60:0.5cm) .. (ga);
      \node[dotnode] (ga2) at (-0.7,0) {};
      \node at (-0.85,-0.15) {\tiny $\gamma$};
      \draw[->] (-2,1.6)
        .. controls +(0:0.8cm) and +(120:0.5cm) .. (ga2)
        .. controls +(-60:0.8cm) and +(180:1cm) .. (2,-1)
        node[pos=1,right] {$Z$};
      \draw[overline] (0.2,0) -- (1,0);
      \draw[overline] (0,0.5) -- (1,0.5);
      \end{scope}
    \end{tikzpicture}
  \end{center}

  \noindent However, this follows directly from the tensoriality
  of the half-braiding $\gamma$ and the sliding lemma
  \ref{lemma/sliding-lemma}.

\end{proof}

\subsection{Statement \& Proof}

We state our main theorem in this paper and will provide a proof
after a few lemmas.

\begin{theorem}[Main Theorem]
  If $C$ is modular, then the functor $G$ is a factorization of
  the Drinfeld center $Z(C)$. More precisely, $G$ is an inverse
  functor for $F$ witnessed by the natural transformations
  $b, d, p, q$.
\end{theorem}

\noindent To prove the main theorem, we first observe an easy
lemma.

\noindent Note that the following lemma does not assume
modularity.

\begin{lemma}
  Let $X \boxtimes Y$ be an object in $C \boxtimes C^{bop}$. Then
  the morphism
  $$X \boxtimes Y \xrightarrow{q \circ d} X \boxtimes Y$$
  is equal to the identity morphism $id_{X \boxtimes Y}$.
\end{lemma}

\begin{proof}
  We prove the equality by direct computation.
  \begin{equation}
    \begin{split}
      q \circ d & = (\sum_i q_i) \circ (\sum_j d_j) \\
      & = \sum_{i} q_{i} \circ d_{i} \\
      & = \sum_{i} \sum_{k=1}^{|X[i]|} \sum_{r=1}^{|X[i]|} \frac{1}{dim(i)} q_{i}(k) \circ d_{i}(r) \\
      & = \sum_{i} \sum_{k=1}^{|X[i]|} \frac{1}{dim(i)} q_{i}(k) \circ d_{i}(k) \\
      & = \frac{1}{dim(\Omega)} \sum_{i} \sum_{k=1}^{|X[i]|} \frac{1}{dim(i)}
      \left(
      \begin{tikzpicture}
        \node (Y) at (0,0) {$Y$};
        \node (X) at (0,1.6) {$X$};
        \node at (0,0.8) {$\boxtimes$};
        \node (Y2) at (5,0) {$Y$};
        \node (X2) at (5,1.6) {$X$};
        \node at (5,0.8) {$\boxtimes$};
        \node[emptynode] (y) at (1.7,-0.6) {};
        \node[emptynode] (i8) at (1.7,0) {};
        \node[emptynode] (i) at (1.7,0.6) {};
        \node[emptynode] (y2) at (3.3,-0.6) {};
        \node[emptynode] (i82) at (3.3,0) {};
        \node[emptynode] (i2) at (3.3,0.6) {};
        \draw (i) to[out=180,in=180] (i8);
        \draw (i2) to[out=0,in=0] (i82);
        \node at (2.2,1) {\tiny $\Omega$};
        \draw[->] (Y) to[out=0,in=180] (y) -- (y2) to[out=0,in=180] (Y2);
        \draw (i8) -- (i82);
        \draw[->] (X) -- (X2)
          node[pos=0.5,above] {$\alpha_i^k \circ \alpha_{i,k}$};
        \draw[midarrow=0.9] (i) -- (i2);
        \node at (3.2,0.8) {\tiny $i$};
        \draw[overline] (2.2,0.3)
          .. controls +(90:0.8cm) and +(90:0.8cm) .. (2.7,0.3)
          .. controls +(-90:0.8cm) and +(-90:0.8cm) .. (2.2,0.3);
        \draw[overline] (2.5,0.6) -- (3,0.6);
        \draw[overline] (2.5,0) -- (3,0);
      \end{tikzpicture} \right) \\
      & = id_{X \boxtimes Y}
    \end{split}
  \end{equation}
  The first pair of sums collapse to a single sum because
  $Hom_{C}(i,j)$ is zero unless $i = j$. The second pair of sums
  collapse to a single sum by the simplicity of $i$ and the
  definition of the pairing between $Hom_{C}(X,i^{\star})$ and
  $Hom_{C}(i^{\star},X)$.
\end{proof}

\begin{lemma}
  Suppose $C$ is modular. Let $X \boxtimes Y$ be an object in
  $C \boxtimes C^{bop}$. Then $(d \circ q)_{X \boxtimes Y}$ is
  equal to the identity isomorphism $id_{GF(X \boxtimes Y)}$.
\end{lemma}
\begin{proof}
  Recall that the image of $X \boxtimes Y$ under $GF$ is
  $$\bigoplus_{i \in O(C)} i^{\star} \boxtimes I_{i},$$
  where $I_{i}$ denotes $I_{i, F(X \boxtimes Y)}$. We will prove
  the equality by direct computation.
  \begin{equation}
    \begin{split}
      d \circ q & = (\sum_i d_i) \circ (\sum_j q_j) \\
      & = \sum_{i} d_{i} \circ q_{i} \\
      & = \sum_{i} \sum_{k=1}^{|X[i]|} \sum_{r=1}^{|X[i]|} \frac{1}{dim(i)} d_{i}(r) \circ q_{i}(k) \\
      & = \sum_{i} \sum_{k=1}^{|X[i]|} \frac{1}{dim(i)} d_{i}(k) \circ q_{i}(k) \\
      & = \frac{1}{dim(\Omega)} \sum_{i} \sum_{k=1}^{|X[i]|}
      \left(
      \begin{tikzpicture}
        \node (I) at (0,0) {$I_i$};
        \node (i8) at (0,1.6) {$i^*$};
        \node at (0,0.8) {$\boxtimes$};
        \node (I2) at (5,0) {$I_i$};
        \node (i82) at (5,1.6) {$i^*$};
        \node at (5,0.8) {$\boxtimes$};
        \node at (0.5,0) {$\subseteq$};
        \node (y) at (1,-0.6) {$Y$};
        \node (x) at (1,0) {$X$};
        \node (i) at (1,0.6) {$i$};
        \node at (4.5,0) {$\twoheadrightarrow$};
        \node (y2) at (4,-0.6) {$Y$};
        \node (x2) at (4,0) {$X$};
        \node (i2) at (4,0.6) {$i$};
        \draw[->] (y) -- (y2);
        \draw[->] (i8) -- (i82);
        \draw (x) -- (x2);
        \draw (i) -- (i2);
        \draw[overline] (2.2,0.3)
          .. controls +(90:0.8cm) and +(90:0.8cm) .. (2.7,0.3)
          .. controls +(-90:0.8cm) and +(-90:0.8cm) .. (2.2,0.3);
        \draw[overline] (2.5,0.6) -- (3,0.6);
        \draw[overline] (2.5,0) -- (3,0);
        \node[dotnode] at (1.7,0) {};
        \node at (1.7,-0.2) {\tiny $\alpha_{i,k}$};
        \node[dotnode] at (3.2,0) {};
        \node at (3.3,-0.2) {\tiny $\alpha_i^k$};
        \end{tikzpicture} \right) \\
      & = id_{GF(X \boxtimes Y)}
    \end{split}
  \end{equation}
  The pairs of sums collapse as in the proof of the last lemma. The
  cut skeins are connected and protected by a $\Omega$-circle by
  lemma \ref{lemma/censorship-of-opacity}.
\end{proof}

\begin{lemma}
  Suppose $C$ is modular. Let $(X,\gamma)$ be an object in
  $Z(C)$. Then $(p \circ b)_{(X,\gamma)}$ is equal to the
  identity isomorphism $id_{(X,\gamma)}$.
\end{lemma}
\begin{proof}
  It is not hard to check that $(p \circ b)_{(X,\gamma)}$ is
  equal to the product of $\frac{1}{dim(\Omega)}$ and the
  following morphism

  \begin{center}
    \begin{tikzpicture}
      \draw[regular] (0.6,0) -- (1.4,0)
        .. controls +(0:0.5cm) and +(0:0.5cm) .. (1.4,0.5)
        -- (0.6,0.5)
        .. controls +(180:0.5cm) and +(180:0.5cm) .. (0.6,0);
      \node at (1.2,0.65) {\tiny $\Omega$};
      \node[dotnode] (ga) at (1,-0.5) {};
      \node at (0.9,-0.7) {\tiny $\gamma$};
      \draw[overline,regular] (ga)
        .. controls +(120:0.9cm) and +(120:0.5cm) .. (1.2,0)
        .. controls +(-60:0.9cm) and +(-60:0.5cm) .. (ga);
      \node at (1.5,-0.8) {\tiny $\Omega$};
      \draw[overline,->] (0,-0.5) -- (ga) -- (2,-0.5);
      \draw[overline,regular] (1,0) -- (1.4,0);
      \end{tikzpicture}
  \end{center}

  \noindent By lemma \ref{lemma/censorship-of-opacity}, the
  horizontal $\Omega$ kills off all nontrivial components in the
  vertical $\Omega$ providing the desired equality.
\end{proof}

\begin{lemma}
  Suppose $C$ is modular. Let $(X,\gamma)$ be an object in
  $Z(C)$. Then $(b \circ p)_{(X,\gamma)}$ is equal to the
  identity isomorphism $id_{FG(X,\gamma)}$.
\end{lemma}
\begin{proof}
  We prove the equality by direct computation.

  \begin{center}
    \begin{tikzpicture}
      \begin{scope}[shift={(0,0)}]
      \draw (0,0) -- (0.8,0)
        .. controls +(0:0.5cm) and +(0:0.5cm) .. (0.8,0.5) -- (0,0.5);
      \draw (3,0) -- (2.2,0)
        .. controls +(180:0.5cm) and +(180:0.5cm) .. (2.2,0.5) -- (3,0.5);
      \node[dotnode] (ga) at (0.5,-0.5) {};
      \node[dotnode] (ga2) at (2.3,-0.5) {};
      \node at (0.4,-0.7) {\tiny $\gamma$};
      \node at (2.2,-0.7) {\tiny $\gamma$};
      \draw[overline] (ga)
        .. controls +(120:0.9cm) and +(120:0.5cm) .. (0.8,0)
        .. controls +(-60:0.9cm) and +(-60:0.5cm) .. (ga);
      \draw[overline] (ga2)
        .. controls +(120:0.9cm) and +(120:0.5cm) .. (2.6,0)
        .. controls +(-60:0.9cm) and +(-60:0.5cm) .. (ga2);
      \draw[overline,->] (0,-0.5) -- (ga) -- (ga2) -- (3,-0.5);
      \draw[overline] (0.5,0) -- (0.8,0)
        .. controls +(0:0.5cm) and +(0:0.5cm) .. (0.8,0.5);
      \draw[overline] (2.4,0) -- (3,0);
      \end{scope}
      \node at (4,0) {$=$};
      \begin{scope}[shift={(5,0)}]
      \draw (0,0) -- (0.8,0)
        .. controls +(0:0.5cm) and +(0:0.5cm) .. (0.8,0.5) -- (0,0.5);
      \draw (3,0) -- (2.2,0)
        .. controls +(180:0.5cm) and +(180:0.5cm) .. (2.2,0.5) -- (3,0.5);
      \node[dotnode] (ga2) at (2.3,-0.5) {};
      \draw[overline] (ga2)
        .. controls +(120:0.9cm) and +(120:0.5cm) .. (2.6,0)
        .. controls +(-60:0.9cm) and +(-60:0.5cm) .. (ga2);
      \node[dotnode] (a1) at (2.22,-0.35) {};
      \node[dotnode] (a2) at (2.4,-0.65) {};
      \node[emptynode] (x) at (0.4,-0.3) {};
      \node[emptynode] (y) at (0.6,-0.7) {};
      \draw[overline] (x)
        .. controls +(120:0.8cm) and +(120:0.5cm) .. (0.8,0)
        .. controls +(-60:0.9cm) and +(-60:0.2cm) .. (y);
      \draw[overline] (x)
        .. controls +(-60:0.1cm) and +(180:1.5cm) .. (a1);
      \draw[overline] (y)
        .. controls +(120:0.1cm) and +(180:1.5cm) .. (a2);
      \draw (ga2) -- (a2);
      \draw[thin_overline={1.5},->] (0,-0.5) -- (ga2) -- (3,-0.5);
      \draw[overline] (0.5,0) -- (0.8,0)
        .. controls +(0:0.5cm) and +(0:0.5cm) .. (0.8,0.5);
      \draw[overline] (2.4,0) -- (3,0);
      \end{scope}
      \node at (9,0) {$=$};
      \begin{scope}[shift={(10,0)}]
      \draw (0,0) -- (0.8,0)
        .. controls +(0:0.5cm) and +(0:0.5cm) .. (0.8,0.5) -- (0,0.5);
      \draw (3.3,0) -- (2.2,0)
        .. controls +(180:0.5cm) and +(180:0.5cm) .. (2.2,0.5) -- (3.3,0.5);
      \node[dotnode] (ga2) at (2.5,-0.5) {};
      \node[emptynode] (a1) at (2,-0.35) {};
      \node[emptynode] (a2) at (2.2,-0.65) {};
      \node[emptynode] (x) at (0.4,-0.3) {};
      \node[emptynode] (y) at (0.6,-0.7) {};
      \node[emptynode] (x2) at (2.2,-0.3) {};
      \node[emptynode] (y2) at (2.4,-0.7) {};
      \draw[overline] (x)
        .. controls +(120:0.8cm) and +(120:0.5cm) .. (0.8,0)
        .. controls +(-60:0.9cm) and +(-60:0.2cm) .. (y);
      \draw[overline] (x)
        .. controls +(-60:0.1cm) and +(180:1.5cm) .. (a1)
        .. controls +(0:0.1cm) and +(-60:0.05cm) .. (x2)
        .. controls +(120:0.6cm) and +(120:0.5cm) .. (2.6,0)
        .. controls +(-60:0.9cm) and +(-60:0.15cm) .. (y2)
        .. controls +(120:0.05cm) and +(0:0.1cm) .. (a2);
      \draw[overline] (y)
        .. controls +(120:0.1cm) and +(180:1.5cm) .. (a2);
      \draw[overline] (ga2)
        .. controls +(120:0.9cm) and +(120:0.5cm) .. (2.8,0)
        .. controls +(-60:0.9cm) and +(-60:0.5cm) .. (ga2);
      \draw[thin_overline={1.5},->] (0,-0.5) -- (ga2) -- (3.3,-0.5);
      \draw[overline] (0.5,0) -- (0.8,0)
        .. controls +(0:0.5cm) and +(0:0.5cm) .. (0.8,0.5);
      \draw[overline] (2.4,0) -- (3.3,0);
      \end{scope}
      \node at (1.5,-2) {$=$};
      \begin{scope}[shift={(2.5,-2)}]
      \draw (1,0.25)
        .. controls +(-90:0.2cm) and +(-90:0.2cm) .. (1.8,0.25);
      \draw[thin_overline={1.5}] (0,0) -- (0.8,0)
        .. controls +(0:0.5cm) and +(0:0.5cm) .. (0.8,0.5) -- (0,0.5);
      \draw[thin_overline={1.5}] (3,0) -- (2,0)
        .. controls +(180:0.5cm) and +(180:0.5cm) .. (2,0.5) -- (3,0.5);
      \node[dotnode] (ga2) at (2.3,-0.5) {};
      \draw[overline] (ga2)
        .. controls +(120:0.9cm) and +(120:0.5cm) .. (2.6,0)
        .. controls +(-60:0.9cm) and +(-60:0.5cm) .. (ga2);
      \draw[thin_overline={1.5},->] (0,-0.5) -- (ga2) -- (3,-0.5);
      \draw[overline] (2.4,0) -- (3,0);
      \draw[thin_overline={1.5}] (1,0.25)
        .. controls +(90:0.2cm) and +(90:0.2cm) .. (1.8,0.25);
      \end{scope}
      \node at (6.5,-2) {$=$};
      \begin{scope}[shift={(7.5,-2)}]
      \draw (0,0) -- (3,0);
      \draw (0,0.5) -- (3,0.5);
      \node[dotnode] (ga2) at (2,-0.5) {};
      \draw[overline] (ga2)
        .. controls +(120:0.9cm) and +(120:0.5cm) .. (2.3,0)
        .. controls +(-60:0.9cm) and +(-60:0.5cm) .. (ga2);
      \draw[thin_overline={1.5},->] (0,-0.5) -- (ga2) -- (3,-0.5);
      \draw[overline] (2,0) -- (3,0);
      \end{scope}
    \end{tikzpicture}
  \end{center}

  The tensoriality of $\gamma$ allows the left circle to attach
  on the right; thus follows the first equality. The sliding
  lemma allows us to slide one strand to the background, and then
  again the tensoriality of $\gamma$ allows detachment; thus
  follows the second equality. Finally, we sear the two strands
  and use lemma \ref{lemma/censorship-of-opacity} to smooth it
  out; thus follows the third equation.
\end{proof}

\section{Discussion \& Prospect}

Using the topological insight from the Crane-Yetter TQFT, we
provided an explicit equivalence between $C \boxtimes C^{bop}$
and $Z(C)$ for modular categories $C$. With the same idea, we
can also provide explicit equivalences (and witnessing natural
isomorphisms) between the categorical centers of higher genera
$Z_{\Sigma}(C)$ \cite{guu/higher-genera-center} and $C^{n}$,
where $\Sigma$ is an oriented surface with $n$ punctures. In
particular, this provides an explicit equivalence between $C$ and
the elliptic Drinfeld center $Z^{el}(C)$ \cite{elliptic--tham}.

We stress again that this only works in the case where $C$ is
modular. This happens for a good reason. Over modular categories,
the Crane-Yetter theory is expected to trivialize to the
Witten-Reshetikhin-Turaev theory by taking boundaries. It is
interesting to investigate the situation where $C$ is not
modular. In fact, this is the motivation of the current paper. We
expect that by measuring how the adjoint functors $F$ and $G$
fail to be an inverse of the other, the difference between both
theories will become clear, leading to a better understanding of
the full power of the Crane-Yetter theory. Moreover, this will
also help understand the structures of the categorical center of
higher genera (note that the Drinfeld center is hard enough).

One way to attack this problem is to look for a general tool in
category theory that measures the failure of the invertibility of
a pair of adjoint functors. Unfortunately, the authors have not
found such tools yet. Another way is to analyze how the
invertibility fails for explicit non-modular categories. Such
genuine premodular categories arise from (super) groups, crossed
modules, and the even part of the semisimplification of
$Rep(U_{q}(\frak{sl}_{2}))$ \cite{q-mackay}.

\section*{Acknowledgement}

It is a great honor of the authors to contribute to such a
beautiful theory. The authors would like to express deep
gratitude to Alexander Kirillov for his guidance.

\printbibliography

\begin{flushright}
  Compiled Time: [\today\,\DTMcurrenttime]
\end{flushright}

\end{document}